\documentclass{article}
\usepackage[utf8]{inputenc}
\usepackage{amsmath}
\usepackage{amssymb}
\usepackage{amsthm}
\usepackage{appendix}
\usepackage{enumerate}
\usepackage{mathrsfs}
\usepackage{scrextend}
\usepackage{mathtools}
\usepackage[hyperfootnotes=false]{hyperref}
\usepackage{cleveref}
\usepackage[multiple]{footmisc}
\usepackage[a4paper, total={6in, 9in}]{geometry}
\usepackage{scalerel,stackengine}
\usepackage{tikz}
\usepackage{float}
\stackMath
\newcommand\reallywidehat[1]{%
\savestack{\tmpbox}{\stretchto{%
  \scaleto{%
    \scalerel*[\widthof{\ensuremath{#1}}]{\kern-.6pt\bigwedge\kern-.6pt}%
    {\rule[-\textheight/2]{1ex}{\textheight}}
  }{\textheight}%
}{0.5ex}}%
\stackon[1pt]{#1}{\tmpbox}%
}
\newcommand{\N}{\mathbb{N}}

\newcommand{\R}{\mathbb{R}}

\newcommand{\lcm}{\mathrm{lcm}}
\newcommand{\ord}{\mathrm{ord}}

\renewcommand{\P}{\mathbb{P}}

\allowdisplaybreaks

\theoremstyle{plain}
\newtheorem{theorem}{Theorem}
\newtheorem{lemma}[theorem]{Lemma}
\newtheorem{proposition}[theorem]{Proposition}
\newtheorem{corollary}[theorem]{Corollary}

\newtheorem{fact}[theorem]{Fact}

\theoremstyle{definition}

\newtheorem{remark}[theorem]{Remark}

\numberwithin{theorem}{section}

\title{The most probable order of a random permutation}
\author{Adrian Beker\footnote{University of Zagreb, Faculty of Science, Department of Mathematics, Zagreb,
Croatia. Email: \nolinkurl{adrian.beker@math.hr}}}
\date{\today}

\begin{document}

\maketitle

\begin{abstract}
    Given positive integers $n$ and $m$, let $p_n(m)$ be the probability that a uniform random permutation of $[n]$ has order exactly $m$. We show that, as $n \to \infty$, the maximum of $p_n(m)$ over all $m$ is asymptotic to $1/n$, the probability of an $n$-cycle. Furthermore, for sufficiently large $n$, we show that the maximum is attained precisely if $m$ is the least positive integer divisible by all positive integers less than or equal to $n-m$. This answers a question of Acan, Burnette, Eberhard, Schmutz and Thomas, originally attributed to work of Erdős and Turán from 1968.
\end{abstract}

\section{Introduction}\label{sec:intro}

Let $\pi_n$ be a permutation chosen uniformly at random from $S_n$, the symmetric group on $n$ letters. Let $\ord(\pi_n)$ denote the order of $\pi_n$, which can be computed as the least common multiple of the lengths of its cycles. Understanding the distribution of $\ord(\pi_n)$ is a fundamental problem in probabilistic group theory. Its study goes back more than a hundred years to the work of Landau \cite{landau}, which established that the maximum of its support, now known as Landau's function, is of the form $e^{(1+o(1))\sqrt{n\log n}}$. Later on, a rather systematic treatment of this subject was undertaken by Erdős and Turán. In a series of works starting in the 1960s, they established a number of results concerning the distribution of (the logarithm of) $\ord(\pi_n)$, including a weak law of large numbers \cite{erdos-turan-groups-i}, a central limit theorem \cite{erdos-turan-groups-iii} and a log-asymptotic for the size of the support \cite{erdos-turan-groups-iv}. For a more complete account of the literature on this and related topics, we recommend the reference \cite{ford}.

While the macroscopic behaviour of $\ord(\pi_n)$ is by now fairly well understood, obtaining local limit results has proved considerably more challenging. In this direction, Acan, Burnette, Eberhard, Schmutz and Thomas \cite{acan-burnette-eberhard-schmutz-thomas} recently studied the so-called collision entropy of $\ord(\pi_n)$. Letting $\pi_n'$ be an independent copy of $\pi_n$, they were interested in estimating the probability that $\ord(\pi_n)$ equals $\ord(\pi_n')$. They proved that, somewhat surprisingly, this quantity is not $O(1/n^2)$, the lower bound coming from the event that $\pi_n$ and $\pi_n'$ are both $n$-cycles (see also \cite{eberhard}). They also established the corresponding upper bound
\begin{equation}\label{eq:equal_order_prob}
    \P(\ord(\pi_n) = \ord(\pi_n')) \leq n^{-2+o(1)}.
\end{equation}
Writing $p_n$ for the probability mass function of $\ord(\pi_n)$, i.e.\ $p_n(m) \vcentcolon= \P(\ord(\pi_n) = m)$ for $m \in \N$, this can be recast as a statement about the $\ell^2$-norm of $p_n$, namely $\lVert p_n\rVert_{\ell^2(\N)} \leq n^{-1+o(1)}$. 

Motivated by a question of Erdős and Turán \cite[p.\ 414]{erdos-turan-groups-iv}, also reiterated by Acan et al.\ \cite[\S6]{acan-burnette-eberhard-schmutz-thomas}, we are instead interested in the maximum probability that $\ord(\pi_n)$ equals a particular value. Writing $M(n) \vcentcolon= \lVert p_n\rVert_{\ell^{\infty}(\N)}$ for this quantity, one readily observes that
\begin{equation}\label{eq:lower_bound_cycle}
    M(n) \geq p_n(n) \geq \P(\pi_n \text{ is an $n$-cycle}) = 1/n.
\end{equation}
In the other direction, since $\lVert p_n\rVert_{\ell^{\infty}(\N)} \leq \lVert p_n\rVert_{\ell^2(\N)}$, the results of \cite{acan-burnette-eberhard-schmutz-thomas} imply a bound of the form
\begin{equation}\label{eq:weaker_bound}
    M(n) \leq n^{-1+o(1)}.
\end{equation}
Our goal is to obtain sharper bounds on $M(n)$. To this end, it will be useful to recall the following setup from \cite{eberhard}. We define
\[K_n \vcentcolon= \bigl\{k \in \{0,1,\ldots,n-1\}\ \mid\ \lcm(1,2,\ldots,k) \mid n-k\bigr\}\]
and note the key property that for $k \in K_n$, the existence of a cycle of length $n-k$ in $\pi_n$ guarantees that $\ord(\pi_n) = n-k$. Note also that $K_n$ always contains $\{0,1\}$ and the prime number theorem implies that $\max K_n \ll \log n$. In particular, if $n$ is large enough, then $\pi_n$ can only contain one cycle of length $n-k$ with $k \in K_n$, and this happens with probability exactly $1/(n-k)$.

Our first result can be regarded as an anticoncentration estimate for $\ord(\pi_n)$ confirming that the lower bound \eqref{eq:lower_bound_cycle} on $M(n)$ is asymptotically tight. As a byproduct, we also obtain a structural description of those $m$ for which $p_n(m)$ is close to its maximum value.

\begin{theorem}
\label{thm:asymp_for_max}
We have the asymptotic $M(n) \sim 1/n$. Moreover, if $n$ is sufficiently large, then any $m$ such that $p_n(m) \geq 1/n$ is of the form $n-k$ for some $k \in K_n$.
\end{theorem}

The second result identifies the mode of the distribution of $\ord(\pi_n)$, i.e.\ the value of $m$ for which $p_n(m)$ attains its maximum, at least when $n$ is large enough. Together with Theorem \ref{thm:asymp_for_max}, this gives an essentially complete answer to the questions raised Erdős and Turán and Acan et al.

\begin{theorem}
\label{thm:equality_case}
For all sufficiently large $n$, we have $p_n(m) = M(n)$ if and only if $m = n - \max K_n$.
\end{theorem}

\begin{remark}
\label{rem:suff_large}
In principle, one could extract from our arguments a bound on how large $n$ needs to be for the above results to hold. However, what one gets is most probability not small enough in order to check the remaining cases by a naive method. Nevertheless, one cannot entirely drop the assumption that $n$ is sufficiently large since numerical evidence shows that counterexamples do exist for small values $n$.
\end{remark}

In order to prove Theorem \ref{thm:asymp_for_max}, the idea is to consider the joint distribution of the order and the number of cycles, and apply different local limit laws according to whether the number of cycles is large, intermediate or small. In doing so, a key difficulty is avoiding the use of lower tail bounds for the number of cycles -- such an approach is unlikely to produce a bound better than \eqref{eq:weaker_bound}. Nevertheless, after making certain refinements, we are able to make use of some of the methods of \cite{acan-burnette-eberhard-schmutz-thomas}. To establish Theorem \ref{thm:equality_case}, it remains to prove a local limit law confirming the prediction that $\P(\ord(\pi_n) = n-k)$ equals $1/(n-k)$ up to a suitably small error. This can be accomplished by adapting some of the existing results in the literature pertaining to the case $k = 0$.

The rest of the paper is organised as follows. In Section \ref{sec:prelim}, we collect some preliminary facts about the distribution of the cycle type of $\pi_n$ that we will need. Section \ref{sec:asymp_for_max} is devoted to the proof of Theorem \ref{thm:asymp_for_max}. Finally, in Section \ref{sec:equality_case}, we present the proof of Theorem \ref{thm:equality_case}.

\bigskip

\noindent\textbf{Notation.} We use Vinogradov asymptotic notation. Given quantities $A$ and $B$, we write $A \ll B$ to mean $A \leq O(B)$, that is to say there is an absolute constant $C > 0$ such that $|A| \leq C|B|$. This is equivalent to the notation $B \gg A$, i.e.\ $B \geq \Omega(A)$. For functions $f, g \colon \N \to \R$, we write $f(n) = o(g(n))$ and $f(n) \sim g(n)$ to mean that $\lim_{n\to\infty}f(n)/g(n) = 0$ and $\lim_{n\to\infty}f(n)/g(n) = 1$ respectively.

\section{Preliminaries}\label{sec:prelim}

It will be useful to recall that the cycle type of a random permutation can be sampled as follows. Given a positive integer $n$, consider the Markov chain $(X_j^{(n)})_{j\geq0}$ with state space $\N_0$, initial distribution $\P(X_0^{(n)} = n) = 1$ and transition probabilities
\[\P(X_{j+1}^{(n)}=u\mid X_j^{(n)}=v) = \begin{cases}\frac{1}{v} & \text{if } u < v\\1 & \text{if } u = v = 0\\0 & \text{otherwise}\end{cases}.\]
Letting $T^{(n)} \vcentcolon= \min\{j\geq 0 \mid X_j^{(n)} = 0\}$ be the time of hitting $0$, we have the following well-known fact (see e.g.\ \cite[p.\ 257-258]{feller}).

\begin{fact}
\label{fact:markov_chain}
The cycle type of $\pi_n$ has the same distribution as the multiset 
\[\{X_j^{(n)}-X_{j+1}^{(n)}\mid 0 \leq j < T^{(n)}\}.\]
In particular, $\ord(\pi_n)$ has the same distribution as $\lcm\{X_j^{(n)}-X_{j+1}^{(n)}\mid 0 \leq j < T^{(n)}\}$.
\end{fact}

Conditioning on the first step of $(X_j^{(n)})_{j\geq0}$, we obtain the following recursive expression for $p_n(m)$. Here and throughout, we use the standard notation $\tau(m)$, $\sigma(m)$ and $\omega(m)$ for the number, sum of positive divisors and the number of prime factors of $m$, respectively. We also recall the well-known fact that $\tau(m) \leq m^{o(1)}$ (see e.g.\ Theorem 2 in Chapter I.5 of \cite{tenenbaum}), which will be used several times in the paper.

\begin{corollary}
\label{cor:recursion}
For any $m, n \in \N$ we have
\[\P(\ord(\pi_n) = m) = \frac{1}{n}\sum_{\substack{0\leq n'<n\\n-n'\mid m}}\P(\lcm(\ord(\pi_{n'}), n-n') = m).\]
In particular, for any $m, n \in \N$ we have
\[\P(\ord(\pi_n) \mid m) \leq \frac{\tau(m)}{n}.\]
\end{corollary} 

We now turn to a lemma that will serve as a key tool in the proof of Theorem \ref{thm:asymp_for_max}. Before stating it, we quickly set up some terminology. For a permutation $\pi \in S_n$, we define $c(\pi)$ to be the number of cycles in $\pi$. For an arbitrary set $I \subseteq \N$, we say $\pi$ is \emph{$I$-restricted} if the length of each cycle in $\pi$ belongs to $I$.

\begin{lemma}
\label{lm:cycle_lengths}
For any $\ell, n \in \N$ and $I \subseteq \N$ we have
\[\P(c(\pi_n) = \ell,\ \text{$\pi_n$ is $I$-restricted}) \leq \frac{\Bigl(\sum_{i\in I}1/i\Bigr)^{\ell-1}}{n(\ell-1)!}.\]
\end{lemma}

Lemma \ref{lm:cycle_lengths} is a special case of \cite[Theorem 1.5]{ford}, a more general local limit law for counts of cycle lengths in disjoint sets -- it follows by taking $r = 2$, $(I_1,I_2) = (I, [n]\setminus I)$ and $(m_1, m_2) = (\ell,0)$. At the same time, it can be obtained by a straightforward generalisation of the argument behind \cite[Lemma 4.1]{acan-burnette-eberhard-schmutz-thomas}, which corresponds to the case $I = [n]$. 

Taking $I$ to be the set of divisors of a given positive integer $m$, we obtain the following corollary, which is effective when $\ell$ is not too small and $m$ is not exceedingly large.

\begin{corollary}
\label{cor:cycle_divisors}
For any $\ell, m, n \in \N$ we have
\[\P(c(\pi_n) = \ell,\ \ord(\pi_n) \mid m) \leq \frac{1}{n(\ell-1)!}\Bigl(\frac{\sigma(m)}{m}\Bigr)^{\ell-1}.\]
\end{corollary}

Finally, we require the following lemma, which will be useful in the regime when the number of cycles is significantly below its expected value and the order is somewhat large. It is based on Fact \ref{fact:markov_chain} and can essentially be read out of the proof of \cite[Lemma 5.1]{acan-burnette-eberhard-schmutz-thomas}.

\begin{lemma}
\label{lm:cycles_and_primes}
For any $\ell, m, n \in \N$ we have
\[\P(c(\pi_n) = \ell,\ m \mid \ord(\pi_n)) \leq \frac{\ell^{\omega(m)}}{m},\]
where $\omega(m)$ is the number of distinct prime factors of $m$.
\end{lemma}
\begin{proof}[Proof sketch.]
    Let $m = \prod_{i=1}^{\omega(m)}p_i^{\alpha_i}$ be the prime factorisation of $m$. Then $m$ divides the order if and only if for each $i \in [\omega(m)]$ there exists a cycle of length divisible by $p_i^{\alpha_i}$. Provided the number of cycles is $\ell$, there are $\ell^{\omega(m)}$ ways to assign a cycle to each prime factor. For any such assignment, the probability of the corresponding divisibility conditions being satisfied does not exceed $1/m$. Indeed, the Markov property implies that, conditional on any earlier divisibility constraints, the probability that the length of a cycle is divisible by all prime powers to which it is assigned is at most the reciprocal of their product. The desired conclusion now follows from the union bound.
\end{proof}

\section{Proof of Theorem \ref{thm:asymp_for_max}}\label{sec:asymp_for_max}

The following proposition is the main stepping stone towards Theorem \ref{thm:asymp_for_max}. Roughly speaking, it says that one can restrict attention to values of the order that are not much larger than $n$.

\begin{proposition}
\label{prop:large_orders}
For any $\varepsilon > 0$, we have
\[\max_{m\geq n^{1+\varepsilon}}p_n(m) = o(1/n).\]
\end{proposition}
\begin{proof}
    Fix $\varepsilon > 0$, let $n$ be sufficiently large in terms of $\varepsilon$ and suppose that $m \geq n^{1+\varepsilon}$. The event that $\pi_n$ has order $m$ can be decomposed into the following three events according to the number of cycles in $\pi_n$:
    \[E_1 \vcentcolon= \{c(\pi_n) \leq C_1\log\log n,\ \ord(\pi_n) = m\},\]
    \[E_2 \vcentcolon= \{C_1\log\log n < c(\pi_n) \leq C_2\log n,\ \ord(\pi_n) = m\},\]
    \[E_3 \vcentcolon= \{c(\pi_n) > C_2\log n,\ \ord(\pi_n) = m\},\]
    where $C_1, C_2 > 0$ are sufficiently large absolute constants. We estimate the probabilities of these events in turn. First, the upper tail bound for the number of cycles (see e.g.\ \cite[Corollary 4.2]{acan-burnette-eberhard-schmutz-thomas}) gives
    \[\P(E_3) \leq \P(c(\pi_n) > C_2\log n) = o(1/n)\]
    provided $C_2$ is large enough. Next, since the order of a permutation is at most the product of the lengths of its cycles, we have $\ord(\pi_n) \leq n^{c(\pi_n)}$. Thus, if $m > n^{C_2\log n}$, then $\P(E_2) = 0$. Otherwise, employing the standard estimate $\sigma(m) \ll m\log\log m$ (see e.g.\ Theorem 5 in Chapter I.5 of \cite{tenenbaum}), we obtain $\sigma(m)/m \leq C\log\log n$ for some absolute constant $C > 0$. Hence, using Corollary \ref{cor:cycle_divisors} and the estimate $k! \geq (k/e)^k$, we obtain
    \[\P(E_2) = \sum_{C_1\log\log n<\ell\leq C_2\log n}\P(c(\pi_n) = \ell,\ \ord(\pi_n) = m) \leq C_2\log n\cdot\frac{1}{n}\Bigl(\frac{2Ce}{C_1}\Bigr)^{\frac{1}{2}C_1\log\log n} = o(1/n)\]
    provided $C_1$ is large enough. Finally, by Lemma \ref{lm:cycles_and_primes} and the standard estimate $\omega(m) \ll \frac{\log m}{\log\log m}$ (see e.g.\ Theorem 3 in Chapter I.5 of \cite{tenenbaum}), for any $\ell \leq C_1\log\log n$ we have
    \[\P(c(\pi_n) = \ell,\ \ord(\pi_n) = m) \leq \exp\Bigl(O\Bigl(\frac{\log m}{\log\log n}\cdot\log\log\log n\Bigr) - \log m\Bigr) \leq \frac{1}{n^{1+\varepsilon/2}}.\]
    By summing over all $\ell$ in this range, it follows that
    \[\P(E_1) = \sum_{\ell\leq C_1\log\log n}\P(c(\pi_n) = \ell,\ \ord(\pi_n) = m) \leq \frac{C_1\log\log n}{n^{1+\varepsilon/2}} = o(1/n),\]
    which concludes the proof.
\end{proof}

\begin{remark}
\label{rem:alt_proof}
Following similar lines as above, one can give a somewhat shorter proof of \eqref{eq:equal_order_prob} than in \cite{acan-burnette-eberhard-schmutz-thomas}. Indeed, by an argument analogous to the estimation of the probabilities of the events $E_2$ and $E_3$, one can obtain
\begin{equation}\label{eq:squared_bound}
    \P(c(\pi_n) \geq L,\ \ord(\pi_n) = m) \leq o(1/n^2),
\end{equation}
where $m \in \N$ is arbitrary and we set $L \vcentcolon= C\log n/\log\log n$ for some large constant $C > 0$. Hence, by dividing into cases according to the number of cycles in $\pi_n$ and $\pi_n'$, one can bound the left-hand side of \eqref{eq:equal_order_prob} by
\[\P(\ord(\pi_n)=\ord(\pi_n'),\ c(\pi_n),c(\pi_n') \leq L) + 2\P(\ord(\pi_n)=\ord(\pi_n'),\ c(\pi_n)\geq L).\]
By conditioning on the order of $\pi_n'$ and using \eqref{eq:squared_bound}, the second term can be seen to be $o(n^{-2})$. On the other hand, the lower tail bound for the number of cycles \cite[Corollary 4.2]{acan-burnette-eberhard-schmutz-thomas} implies that the first term is at most
\[\P(c(\pi_n) \leq L)^2 \leq n^{-2+o(1)},\]
whence \eqref{eq:equal_order_prob} follows.
\end{remark}

We are now ready to prove Theorem \ref{thm:asymp_for_max}. In view of \eqref{eq:lower_bound_cycle}, it suffices to show that, under the assumption that $n$ is sufficiently large and $m$ satisfies $p_n(m) \geq 1/n$, we have $p_n(m) \leq (1+o(1))/n$ and $n - m \in K_n$. In particular, by Proposition \ref{prop:large_orders}, we may assume that $m \leq n^{4/3}$ say. The starting point is an application of Corollary \ref{cor:recursion}. By the first statement, we have
\begin{equation}\label{eq:recursion}
    p_n(m) = \frac{1}{n}\sum_{\substack{0\leq k<n\\n-k\mid m}}\P(\lcm(\ord(\pi_k), n-k) = m),
\end{equation}
and by the second statement (applied with $k$ in place of $n$),
\begin{equation}\label{eq:bound_on_summand}
     \P(\lcm(\ord(\pi_k), n-k) = m) \leq \frac{\tau(m)}{k}
\end{equation}
whenever $0 < k < n$. It follows that the total contribution of all $k \geq n^{1/2}$ to the right-hand side of \eqref{eq:recursion} is at most
\[\frac{1}{n}\cdot\tau(m)\cdot\frac{\tau(m)}{n^{1/2}} = \frac{\tau(m)^2}{n^{3/2}} \ll n^{-4/3}.\]
On the other hand, out of those $k$ that are less than $n^{1/2}$, at most one contributes to $p_n(m)$. Indeed, if this were not the case, then $m$ would have two divisors in the interval $(n-n^{1/2},n]$, call them $d_1 < d_2$. This would lead to a contradiction since
\[\lcm(d_1,d_2) = \frac{d_1d_2}{\gcd(d_1,d_2)} \geq \frac{d_1d_2}{d_2-d_1} \geq \frac{(n-n^{1/2})^2}{n^{1/2}} > m.\]
Thus, $m$ has a unique divisor $d \in (n-n^{1/2},n]$ and we have
\[p_n(m) \leq \frac{1}{n}\P(\lcm(\ord(\pi_{n-d}), d) = m) + O(n^{-4/3}).\]
In particular, we have $p_n(m) \leq (1+o(1))/n$. Moreover, since $m$ was assumed to satisfy $p_n(m) \geq 1/n$, it follows that
\begin{equation}\label{eq:small_prob}
    \P(\lcm(\ord(\pi_{n-d}), d) \neq m) \ll n^{-1/3}.
\end{equation}
We contend that this forces $d$ to be divisible by all positive integers less than or equal to $n-d$. Indeed, suppose this does not hold. Then there exists a prime $p$ such that the largest power of $p$ not exceeding $n-d$, call it $q$, doesn't divide $d$. In particular, by maximality of $q$, we have 
\begin{equation}\label{eq:maximality}
    q^2 \geq pq > n-d.
\end{equation}
Let $E$ be the event that the order of $\pi_{n-d}$ is divisible by $q$. Then on $E$, the $p$-adic valuation of $\lcm(\ord(\pi_{n-d}), d)$ equals that of $q$, and on $E^c$, it is strictly less than that of $q$. Consequently, at least one of $E$, $E^c$ is contained in the event on the left-hand side of \eqref{eq:small_prob}, so $\min(\P(E),\P(E^c)) \ll n^{-1/3}$. But by \cite[Lemma 1]{erdos-turan-groups-ii}, we have the exact expression
\[\P(E^c) = \prod_{j=1}^{\lfloor (n-d)/q\rfloor}\Bigl(1-\frac{1}{jq}\Bigr).\]
Hence, using \eqref{eq:maximality}, we obtain the approximation
\begin{equation}\label{eq:approx_prob}
    \frac{1}{q} \leq \P(E) \leq \sum_{j=1}^{\lfloor (n-d)/q\rfloor}\frac{1}{jq} \ll \frac{\log q}{q}.
\end{equation}
Note that by \eqref{eq:small_prob}, we certainly have
\[\P(\lcm(\ord(\pi_{n-d}),d) = m) \geq \frac{1}{2},\]
so \eqref{eq:bound_on_summand} implies $n-d \leq 2\tau(m)$. This means that $q \ll n^{1/4}$ say, whence the lower bound \eqref{eq:approx_prob} implies $\P(E) \gg n^{-1/4}$. Thus, we cannot have $\P(E) \ll n^{-1/3}$, so the only remaining option is $\P(E^c) \ll n^{-1/3}$. In view of the upper bound \eqref{eq:approx_prob}, this means that $q \ll 1$. Hence, by \eqref{eq:maximality}, we also have $n-d \ll 1$. But then necessarily $\P(E^c) = 0$, which is absurd since $\P(\ord(\pi_{n-d}) = 1) > 0$. Therefore, the claim follows, so in particular $d = m$. In other words, we conclude that $m = n-k$ for some $k \in K_n$, thereby completing the proof.

\section{Proof of Theorem \ref{thm:equality_case}}\label{sec:equality_case}

Theorem \ref{thm:equality_case} follows by combining Theorem \ref{thm:asymp_for_max} and the following proposition, which gives accurate control on the point probabilities $\P(\ord(\pi_n) = n-k)$ for $k \in K_n$.

\begin{proposition}
\label{prop:exact_order}
For any $k \in K_n$ we have
\[\P(\ord(\pi_n) = n-k) = \frac{1}{n-k} + \eta(n,k) + O(n^{-3+o(1)}),\]
where we define
\[\eta(n,k) \vcentcolon= \begin{cases}0 & \text{if } k \in \{0,1\} \text{ or } 2^{\lfloor\log_2k\rfloor+1} \mid n-k\\\frac{2^{1-\lfloor\log_2k\rfloor}}{(n-k)^2} & \text{otherwise}\end{cases}.\]
\end{proposition}
\begin{proof}
    The proof is a relatively straightforward adaptation of the arguments of Warlimont \cite{warlimont}, which deal with the case $k = 0$.\footnote{In fact, Warlimont considers the probability that the order divides $n$ instead of being exactly equal to $n$, but this distinction is not significant.} Hence, we will be fairly brief on the details. As in \cite{warlimont}, we start by using Cauchy's formula \cite[Theorem 1.2]{ford} to express
    \[\P(\ord(\pi_n) = n-k) = \frac{1}{n-k} + \sum_{\substack{m,m_1,\ldots,m_r \in \N_0\\m+\sum_{j=1}^{r}m_jd_j = n\\\lcm\{d_j\mid j\in[r],\ m_j > 0\} = n-k}}\frac{1}{m!}\prod_{j=1}^{r}\frac{1}{m_j!d_j^{m_j}},\]
    where $1 < d_1 < \ldots < d_r < n-k$ are the divisors of $n-k$. For each $i \in \N_0$, we let $T_i$ be the total contribution of the terms satisfying $\sum_{j=s+1}^{r}m_j = i$, where $s$ is the number of $j \in [r]$ for which $d_j < n^{1-\delta}$, and $\delta > 0$ is some small parameter. One can then proceed in the same way as in \cite{warlimont} to bound
    \[\sum_{i=3}^{\infty}T_i \ll (\tau(n-k)n^{\delta-1})^3.\]
    Furthermore, in analogy to \cite{warlimont}, one can establish that 
    \[T_0 \leq F(n,k), \quad T_1 \leq \tau(n-k)n^{\delta-1}F(n,k),\] 
    where we define
    \[F(n,k) \vcentcolon= n\sum_{m\geq A(n,k)}\frac{1}{m!} + 2^{-B(n,k)}\tau(n-k)\exp(\tau(n-k)),\]
    \[A(n,k) \vcentcolon= \frac{n}{6\tau(n-k)},\quad B(n,k) \vcentcolon= \frac{n^{\delta}}{6\tau(n-k)}.\]
    In a similar vein, the total contribution to $T_2$ of all terms apart from those with
    \begin{equation}\label{eq:special_terms}
        d_r = \frac{n-k}{2},\quad m_r = 2, \quad m_{s+1} = \ldots = m_{r-1} = 0
    \end{equation}
    is at most $O((\tau(n-k)n^{\delta-1})^2F(n,k))$. It therefore remains to show that the contribution of the terms satisfying \eqref{eq:special_terms} is precisely $\eta(n,k)$. Indeed, if $n$ is large enough, then we have $\tau(n-k) \leq n^{\delta/3}$ and hence 
    \[A(n,k) \geq \frac{1}{6}n^{1-\delta/3}, \quad B(n,k) \geq \frac{1}{6}n^{2\delta/3}.\] 
    Consequently, we may bound
    \[\sum_{i=3}^{\infty}T_i \ll n^{-3+4\delta}\]
    and also
    \[F(n,k) \ll \exp\Bigl(-\frac{1}{6}n^{1-\delta/3}\Bigr) + \exp\Bigl(-\frac{\log2}{6}n^{2\delta/3}+n^{\delta/3}+\frac{\delta}{3}\log n\Bigr),\]
    which is certainly $O(n^{-3})$. Since $\delta > 0$ is arbitrary, we obtain an error term of the desired form.
    
    To finish the proof, we carefully examine the terms that satisfy \eqref{eq:special_terms}. For such terms, we have $m+\sum_{j=1}^{s}m_jd_j = k$. In order to have $\lcm\{d_j \mid j\in [r],\ m_j > 0\} = n-k$, there must exist $j \in [s]$ such that $m_j > 0$ and $\nu_2(d_j) = \nu_2(n-k)$, where $\nu_2$ denotes $2$-adic valuation. For this to be possible, we need $\nu_2(n-k)$ to be equal to the maximum of $\nu_2(t)$ over all $t \in [k]$. In particular, if $k \in \{0,1\}$ or $\nu_2(n-k) > \lfloor\log_2k\rfloor$, this is not possible. Otherwise, the terms of interest are precisely those which in addition to \eqref{eq:special_terms} satisfy $m_j=1$ for the unique $j \in [s]$ such that $d_j = 2^{\lfloor\log_2k\rfloor}$. Their total contribution is easily seen to be
    \[\frac{1}{2\bigl(\frac{n-k}{2}\bigr)^2\cdot2^{\lfloor\log_2k\rfloor}} = \frac{2^{1-\lfloor\log_2k\rfloor}}{(n-k)^2},\]
    as desired.
\end{proof}

To prove Theorem \ref{thm:equality_case}, assume $n$ is sufficiently large and let $k_0 \vcentcolon= \max K_n$. By Theorem \ref{thm:asymp_for_max}, it suffices to prove that $p_n(n-k_0) > p_n(n-k)$ for all $k \in K_n \setminus \{k_0\}$.
Hence, by Proposition \ref{prop:exact_order}, it is enough to show that
\begin{equation}\label{eq:final_ineq}
    \frac{k_0-k}{(n-k_0)(n-k)} + \eta(n,k_0)-\eta(n,k) \geq \frac{1}{(n-k)^2}.
\end{equation}
If $k \in \{0,1\}$, then $\eta(n,k) = 0$, so \eqref{eq:final_ineq} certainly holds. Hence, we may assume that $k \geq 2$, so in particular $\eta(n,k) \leq 1/(n-k)^2$. Since $\lcm(1,2,\ldots,k)$ divides both $n-k$ and $n-k_0$, it must divide $k_0-k$. Therefore, $k_0 - k \geq 2$, so the left-hand side  of \eqref{eq:final_ineq} is at least
\[\frac{2}{(n-k_0)(n-k)} - \frac{1}{(n-k)^2} > \frac{1}{(n-k)^2},\]
as desired.

\begin{remark}
\label{rem:sharp_asymp}
As a consequence of Theorem \ref{thm:equality_case} and Proposition \ref{prop:exact_order}, one obtains the more refined asymptotic 
\[M(n) = \frac{1}{n} + O\Bigl(\frac{\log n}{n^2}\Bigr).\]
The error term here is best possible up to constants, as can be seen by considering $n$ of the form $\lcm(1,2,\ldots,k)+k$ for $k \in \N$.
\end{remark}

\bigskip

\noindent\textbf{Acknowledgements.} This work was supported by the Croatian Science Foundation under the project number HRZZ-IP-2022-10-5116 (FANAP). The author would like to thank Rudi Mrazović for reading an earlier draft of this paper and Sean Eberhard for many useful comments and remarks.

\bibliographystyle{plain}
\bibliography{references}

\end{document}